\documentclass[11pt, reqno]{amsart}
\usepackage{amsmath, amsthm, amscd, amsfonts, amssymb, graphicx, color}
\usepackage[colorlinks=true, linkcolor=blue,urlcolor=blue]{hyperref}
\input xy \xyoption{all}
\usepackage[all]{xy}
\usepackage{layout}
\usepackage{tikz}

\newtheorem{theorem}{Theorem}[section]
\newtheorem{lemma}[theorem]{Lemma}
\newtheorem{proposition}[theorem]{Proposition}
\newtheorem{corollary}[theorem]{Corollary}
\theoremstyle{definition}
\newtheorem{definition}[theorem]{Definition}

\theoremstyle{remark}
\newtheorem{remark}[theorem]{Remark}
\numberwithin{equation}{section}

\newcommand{\N}{\mbox{$\mathbb{N}$}}

\begin{document}
\setcounter{page}{1}

\title[Prime and maximal ideals in Hurwitz polynomial rings]{\textbf{Prime and maximal ideals in Hurwitz polynomial rings}}

\author[Ali Shahidikia]{Ali Shahidikia}
\address{Department of Mathematics, Dezful branch, Islamic Azad University, Dezful, Iran}
\email{\textcolor[rgb]{0.00,0.00,0.84}{ali.Shahidikia@iaud.ac.ir \,\
		a.shahidikia@gmail.com
}}

\begin{abstract}
In this paper we study prime and maximal ideals in a Hurwitz polynomial ring $hR$. It is well-known that to study many questions we may assume R is prime and consider just $R$-disjoint ideals.  We give a characterization for an $R$-disjoint ideal to be prime. We study conditions under which there exists an $R$-disjoint ideal which is a maximal ideal and when this is the case how to determine all such maximal ideals. 
maximal ideal to be generated by polynomials of minimal degree.
\end{abstract}

\maketitle
\section{ Introduction}
The study of formal power series rings has garnered noteworthy attention and has been found to be essential in a multitude of fields, particularly in differential algebra. In \cite{kie}, Keigher introduced a variant of the formal power series ring and studied its categorical properties in detail. This ring was later named Hurwitz series ring.
There are many interesting applications of the Hurwitz series ring in differential algebra as Keigher showed in \cite{kieg,kiegh}.\\


Throughout this article, $R$ denotes an associative ring with unity. We denote $H(R)$, or simply $HR$,  the Hurwitz series ring over a ring $R$ whose elements are the functions $f:\N\rightarrow R$, where $\N$ is the set of all natural numbers, the addition is defined as usual and the multiplication given by 
\begin{equation*}
	(f g)(n)= \sum_{k=0}^n \binom{n}{k} f(k)g(n - k)
	\text{ for  all } n \in \N,
\end{equation*}
where $\binom{n}{k}$ is the binomial coefficient.\par


Define the mappings $h_n:\N\rightarrow R$, $n\geq1$ via $h_n(n-1)=1$ and $h_n(m)=0$ for each $ m \in \N\setminus\{n-1\}$ and $h'_r:\N\rightarrow R$, $r\in R$ via $h'_r(0)=r$ and $h'_r(n)=0$ for each $n\in \N\setminus\{0\}$. It can be easily shown that $h_{1}$ is the unity of $HR$.

For $f\in HR$, the support of $f$, denoted by $supp(f)$, is the set $\{i\in\N\mid  f(i)\not=0\}$. The minimal element in $supp(f)$ is denoted by $\Pi(f)$, and $\Delta(f)$ denotes the greatest element in $supp(f)$ if it exists. Define $R'=\{h'_r\mid r\in R\}$. $R'$ is a subring of $HR$ which is isomorphism to $R$. For any nonempty subset $A$ of $R$ define $A'=\{h'_r\mid r\in A\}.$ If $A$ is an ideal of $R$, then $A'$ is an ideal of $R'$.\par

The ring $hR$ of Hurwitz polynomials over a ring $R$ is the subring of $HR$ that consists elements of the form $f\in HR$ with $\Delta(f)<\infty$ (see \cite{M, m2, m1, v2}).\\

Consider the ring $hR$, if $P$ is a prime ideal of $hR$, by factoring out the
ideals $P\cap R$ and $h(P \cap R)$ from $R$ and $hR$, respectively, we may assume that $R$ is prime and $P\cap R = 0$. That is why we will assume here that $R$ is a prime ring. Then $hR$ is also prime. A non-zero ideal (resp. prime ideal) $P$ of $hR$ with $P \cap R = 0$ will be called an $R$-\textit{disjoint ideal} (resp. \textit{prime ideal}).  

The main purpose of this paper is to study maximal ideals of $hR$. There are two
interesting questions concerning maximal ideals that we want to consider here. First, determine all the prime ideals L of $R$ such that there exists a maximal ideal $M$ of $hR$ with $M \cap R=L$.  As we said above, by factoring out $L$ and $hL$ from $R$ and $hR$,
respectively, we may assume that $L = 0$. Second, assume that there exists a maximal ideal
$M$ of $hR$ which is $R$-disjoint. Then determine all these ideals. In Section \ref{sec3} of this paper
we study these questions.



From now on $R$ is a prime ring with unity. If $I$ is an $R$-disjoint ideal of $hR$, then by $\rho(I)$ (resp. $\tau(I)$) we denote the ideal of $R$ consisting of $0$ and all the leading coefficients of all the polynomials (resp. polynomials of minimal degree) in $I$.  The
minimality of $I$ is defined by $\text{Min}(I) = \text{min}\{\Delta(f)~ |~f\in I\setminus \{0\}\}$. We denote by $Z = Z(R)$ the center of $R$.

We point out that ideal means always two-sided ideal and $R$-disjoint maximal ideal means maximal ideal which is $R$-disjoint.
\section{Prime ideals}
\smallskip
First we recall the following lemma.
\begin{lemma}\label{1.1}
	Let $P$ be an $R$-disjoint ideal of $hR$. The following are equivalent:
	\begin{enumerate}
		\item  $P$ is a prime ideal of $hR$;
		\item $P$ is maximal in the set of $R$-disjoint ideals of $hR$.
	\end{enumerate}
\end{lemma}
\begin{proof}
	The proof is straightforward.
\end{proof}
We define the following set:
\[\Gamma=\{f\in hR~|~\Delta(f)\geq1 \text{~and~} h'_ah'_rf=fh'_rh'_a, \text{~for every~} r\in R, \text{~where~} a=f(\Delta(f))\}.\]
For $f\in \Gamma$ with $a=f(\Delta(f))$ we put
\[[f]=\{g\in hR~|~\text{~there exists~} 0\not=J\vartriangleleft R \text{~such that~} gJ'h'_a\subseteq hRf\}.\]
One can see that $[f]$ is an $R$-disjoint ideal of $hR$.
An ideal of this type is said to be a \textit{(principal) closed ideal} of $hR$. Actually,  $[f]$ is the unique
closed ideal of $hR$ containing $f$ and satisfying $\text{Min}([f]) = \Delta(f)$.

\begin{definition}\label{1.2}	We say that a polynomial $f\in \Gamma$ is \textit{completely irreducible} in $\Gamma$ (or
	$\Gamma$-completely irreducible) if the following condition is satisfied:\\	
	
	If there exist $b \in R$, $g \in \Gamma$ and $h\in hR$ such that $0\not=fh'_b=hg$, then $\Delta(g)=\Delta(f)$.
\end{definition}
This definition is symmetrical. In fact, we have the following result.
\begin{lemma}\label{1.3}
A polynomial $f\in \Gamma$ is $\Gamma$-completely irreducible if and only if for every $b \in R$, $h \in \Gamma$ and $g\in hR$ such that $0\not=fh'_b=hg$ we necessarily have $\Delta(h)=\Delta(f)$.	
\end{lemma}
\begin{proof}
	If $f\in \Gamma$ and $fh'_b\not=0$, $b\in R$, we have $fh'_b\in \Gamma$ and $[fh'_b]=[f]$. In fact, assume that $ab=0$, where $a=f(\Delta(f))$. Then $h'_ah'_rfh'_bfh'_rh'_ah'_b=0$, for every $r\in R$. Thus, $fh'_b=0$ since $R$ is prime, a contradiction. Hence $\Delta(fh'_b)=\Delta(f)$ and $[fh'_b]=[f]$.
	
	Suppose that there is $b\in R$, $h\in \Gamma$ and $g\in hR$ such that $0\not=fh'_b=hg$ with $\Delta(h)<\Delta(f)$. Then $[f]=[fh'_b] \subset[h]$. Therefore there is $0\not=J\vartriangleleft R$ such that $fJh'_c\subseteq hRh$, where $c=h(\Delta(h))$. Take $d\in J$ and $p\in hR$ with $0\not=fh'_dh'_c=ph$. This shows that $f$ is not $\Gamma$-completely irreducible. So the result holds in one direction. The converse can be proved in a similar way using an obvious symmetric version of the definition of $[f]$.
	\end{proof}

	Now we are in position to prove the main result of this section.
	\begin{theorem}\label{1.4}
		Let $P$ be an $R$-disjoint ideal of $hR$. Then the following conditions are equivalent:
		\begin{enumerate}
			\item $P$ is prime;
			\item  $P$ is closed and every $f\in P$ with $\Delta(f) = \text{Min}(P)$ is completely irreducible in $\Gamma$;
			\item $P$ is closed and there exists $f\in P$ with $\Delta(f) = \text{Min}(P)$ which is completely
			irreducible in $\Gamma$.
		\end{enumerate}
	\end{theorem}
	\begin{proof}
		(i)$\Rightarrow$(ii).  If $P$ is prime, then $P$ is closed. Assume that $f]in P$, $\Delta(f)=\text{Min}(P)$ and $0\not=fh'_b=hg$, for $b\in R$ and $g\in \Gamma$. Hence $P=[f]=[fh'_b]\subseteq[g]$. Thus $P=[g]$ by Lemma \ref{1.3}. Consequently $\Delta(g)=\text{Min}(P)=\Delta(f)$.
		
		(iii)$\Rightarrow$(i).  Assume that $P = [f]$, where $f$is a $\Gamma$-completely irreducible polynomial. If $P$
		is not prime there is a closed ideal $[g]$ of $hR$, $g \in \Gamma$, such that $P\subset [g]$. It follows that
		$\Delta(g) < \Delta(f)$ and there is $0\not=J \vartriangleleft R$ such that $fJ'h'_c \subseteq hRg$, where $c = g(\Delta(g))$. Now we get a
		contradiction as in the proof of Lemma \ref{1.3}.
	\end{proof}
	It can easily be checked that if $R$ is a commutative domain and $F$ is the field of
	fractions of $R$, then a polynomial $f\in hR$ with $\Delta(f)\geq1$ is completely irreducible in $\Gamma$ if
	and only if $f$is irreducible in $hF$.
	
	More generally, note that if $f\in hZ$, $Z$ the center of $R$, and $\Delta(f)\geq1$, then $f\in \Gamma$. We
	have
	\begin{corollary}\label{1.5}
	Suppose that $f \in hZ$. Then the following are equivalent:
	\begin{enumerate}
		\item $f$ is completely irreducible in $\Gamma$;
		\item $f$ is irreducible in $hC$, where $C$ is the extended centroid of $R$.
	\end{enumerate}
	\end{corollary}
	\begin{proof}
		(i)$\Rightarrow$(ii).
		Assume that $f = gh$, for $g, h \in hC$,$\Delta(g)<\Delta(f)$ and $\Delta(h)<\Delta(f)$. Then $[f]\subset hQg\cap hR$, where $Q$ is the maximal right quotient ring of $R$. This is a contradiction by Lemma \ref{1.1} because $[f]$ is prime.
		
		(ii)$\Rightarrow$(i).
		Since $f\in \Gamma$ there is a monic polynomial $f_0\in hC$ such that $[f]=hQf_0\cap hR$. Then $f=f_0h'_c$, for $c=f(\Delta(f))\in Z\subseteq C$. Thus $f_0$ is irreducible in $hC$ by the assumption and so $[f]$ is prime \cite[Corollary 2.7]{B1}. Consequently $f$is $\Gamma$-completely
		irreducible by Theorem \ref{1.4}.
	\end{proof}
	\begin{remark}\label{1.6}
	Let $R$ be a unique factorization commutative domain and
	$f\in hR$. Then $f$ is completely irreducible in $\Gamma$ if and only if $f$ is irreducible in $hR$.
	To prove this fact we first recall that a polynomial $g \in hR$ is said to be \textit{primitive} if
	the greatest common divisor of $g(i)$, $i=0,1,2,\ldots$ is $1$.  Also, a primitive polynomial is
	irreducible if and only if it is irreducible in $hF$, where $F$ is the field of fractions of $R$.
	
	Now, if $f$ is completely irreducible in $\Gamma$, then it is clearly irreducible in $hR$. Conversely, if $f$is irreducible in $hR$ and $f$ is primitive, then $f$ is irreducible in $hF$ and
	we apply Corollary \ref{1.5}. In general, there are $d \in R$ and a primitive polynomial $g\in hR$ such that $f=h'_dg$. Since $f$ is irreducible so is $g$. Hence $g$ is completely irreducible in $\Gamma$ and consequently so is $f$.
	\end{remark}
	\section{Maximal ideals}\label{sec3}
	As we said in the introduction, there are two interesting
	questions concerning maximal ideals that we want to consider here. First, determine all
	the prime ideals $L$ of $R$ such that there exists a maximal ideal $M$ of $hR$ with $M\cap R = L$.
	By factoring out convenient ideals we may assume that $L = 0$. Second, determine all the
	$R$-disjoint maximal ideals of $hR$ in case the set of all these ideals is not empty.
	
	We extend the terminology used for commutative rings \cite{B2}. If $R$ is a prime ring, then the intersection of all the non-zero prime ideals of $R$ is called the 
	\textit{pseudo-radical} of $R$ and
	is denoted by $\textrm{ps}(R)$.
	
	We begin this section with the following extension of \cite[Lemma 3]{B3}.
	\begin{lemma}\label{2.1}
	Let $P$ be an $R$-disjoint prime ideal of $hR$ and $L$ be a non-zero	prime ideal of $R$. If $\rho(P) \nsubseteq L$, then $(P + hL) \cap R' = L'$.	
	\end{lemma}
	\begin{proof}
		If $h_2\in P$, then $P=(h_2)(hR)$ and for $h'_r\in   (P+hL)\cap R'$ we easily obtain $h'_r\in L'$. So we may assume $h_2\not\in P$.
assume that there exists $r\in  R\setminus L$ such that $h'_r = f_1 + f_2$, for some $f_1in P$ and $f_2\in hL$. It follows that there exists	 $g\in P$ with $g(0)\not\in L$ and $g(i)\in L$, $i=1,2,\ldots,\Delta(g)$.  Take such a $g$ of minimal degree with respect to these conditions and choose a polynomial $f\in P$ of minimal degree with respect to $f(n)\not\in L$, $n=\Delta(f)$.	

By the assumption there exists $c \in R$ such that $g(0)cf(n)\not\in  L$. If $\Delta(g)\geq n$ we have $gh'_ch'_{f(n)}-h'_{g(m)}h'_ch_{m-n+1}f\in P$, which contradicts the minimality of $\delta(g)$. In case $\Delta(g)> n$ put
$h=h'_(f(0))h'_cf-gh'_ch'_{g(0)}\in P$. Then $h=g'h_2$, where $g'\in hR$ with $g'(n-1)=f(0)cg(n)\not\in L$. Since $P$ is prime and $h_2\not\in P$ we have $g'\in P$, contradicting the minimality of $n$.
	\end{proof}
	\begin{corollary}\label{2.2}
		Let $M$ be a maximal ideal of $hR$  with $M \cap R = 0$. Then $0\not=\rho(M)\subseteq \textrm{ps}(R)$.		
	\end{corollary}
	\begin{proof}
		Since $M\not=0$ we have $\rho(M)\not= 0$. If $L$ is a non-zero prime ideal of $R$, we have
		$M + hL= hR$. So $\rho(M) \subseteq L$ by Lemma \ref{2.1}.
	\end{proof}
	Corollary \ref{2.2} shows that if there exists an $R$-disjoint maximal ideal of $hR$, then
	$\textrm{ps}(R)\not= 0$. The converse is not true in general.
	\begin{proposition}\label{2.3}
	Let $R$ be a prime ring such that every non-zero ideal of $R$ contains a central element. Then there exists an $R$-disjoint maximal ideal of $hR$ if and only if
	$\textrm{ps}(R)\not=0$.
	\end{proposition}
	\begin{proof}
		Suppose that $\textrm{ps}(R)\not=0$ and take $0\not=c\in Z\cap \textrm{ps}(R)$. Put $f=ch_2+h_1$. Then $f\in \Gamma$ and $[f]$ is an $R$-disjoint prime ideal of $hR$. If $I$ is a maximal ideal with $I\supset [f]$ we have $I\cap R'\not=0$. Hence $c\in \textrm{ps}(R)\subseteq I\cap R'$ and $h_1\in I$, a contradiction. Thus $[f]$ is a maximal ideal. The proof is complete by Corollary \ref{2.2}.
	\end{proof}
	Now we give a more general criterion. Denote by $h_1+h_2h\textrm{ps}(R)$ the set of all the polynomials $f\in hR$ with $f(0)=1$ and $f(i)\in \textrm{ps}(R)$ for $i=1,\ldots,\Delta(f)$.
	\begin{proposition}\label{2.4}
		Let $M$ be an $R$-disjoint maximal ideal of $hR$. Then one of the following possibilities occurs:
		\begin{enumerate}
			\item $h_2\in M$, $R$ is simple and $M=h_2hR$;
			\item  $h_2\not\in M$ and $M\cap (h_1+h_2h\textrm{ps}(R))\not=\emptyset$.
		\end{enumerate}
	\end{proposition}
	\begin{proof}
		If $h_2\in M$, then (i) follows. suppose that $h_2\not\in M$. Then $h_1+h_2h\textrm{ps}(R)\nsubseteq M$ and so $M+h_2h\textrm{ps}(R)=hR$. Thus there is $f\in M$ and $g\in \textrm{ps}(R)$ such that $f+h_2g=h_1$. Consequently $f\in h_1+h_2h\textrm{ps}(R)$ and we are done.
	\end{proof}
	\begin{corollary}\label{2.5}
		Let $R$ be a prime ring which is not simple and let $M$ be a
		prime ideal of $hR$ with $M \cap R = 0$. Then $M$ is a maximal ideal if and only if $M\cap(h_1+h_2h\textrm{ps}(R))\not=\emptyset$.
	\end{corollary}
	\begin{proof}
		suppose that $f\in M$ with $f(0)=1$ and $f(i)\in \textrm{ps}(R)$ for $i=1,\ldots,\Delta(f)$. If $I$ is a maximal ideal of $hR$ such that $I\supset M$ we obtain $h_1\in I$ as in Proposition \ref{2.3}. Then $M$ is a maximal ideal. The rest is clear.
	\end{proof}
	The intersection of a finite family of closed ideals is closed \cite[Corollaries 3.4 and
	3.5]{B3}. So for any $f\in hR$ with $\Delta(f)\geq l$ there exists the smallest closed ideal of $hR$ containing $f$. We denote this ideal again by $[f]$. If there is no closed ideal which contains $f$ we put $[f]=hR$.
	\begin{corollary}\label{2.6}
	There exists an $R$-disjoint maximal ideal of $hR$ if and only if either 	$R$ is simple or there exists $f\in h_1+h_2\textrm{ps}(R)$ such that $[f]\not=hR$.	
	\end{corollary}
	\begin{proof}
		If there exists a maximal ideal $M$ of $hR$ with $M \cap R = 0$ we have that either
		$h_2 \in M$ and $R$ is simple or there exists $f\in M \cap (h_1 + h_2 \textrm{ps}(R))$. Thus $[f] \subseteq M \not= hR$.
		
			Conversely, if $R$ is simple, then $h_2hR$ is a maximal ideal. If $R$ is not simple we
			choose a polynomial $f\in  h_1 + h_2 \textrm{ps}(R)$ such that $[f]\not=hR$. Then there exists an $R$-disjoint ideal $M$ which is maximal with respect to $[f]\subseteq M$. Hence $M$ is a maximal ideal
			by Corollary \ref{2.5}.
	\end{proof}
	\begin{remark}\label{2.7}
The set of all the $R$-disjoint maximal ideals $M$ of$hR$ can now be determined as follows. If there is no $f\in  h_1 + h_2 \textrm{ps}(R)$ with $[f]\not=hR$, then $M=\emptyset$. Assume that there exists $f\in  h_1 + h_2 \textrm{ps}(R)$ with $[f]\not=hR$. Then for such a polynomial $f$ there is a uniquely determined finite family of $R$-disjoint prime ideal $P_{1f}, P_{2f},\ldots,P_{n_ff}$ such that $[f]=\bigcap_i[P_{if}^{e_i}]$, where $e_i\geq1$ \cite[Theorem 3.1]{B2}. Then $M=\{P_{if}\}$, where $f\in  h_1 + h_2 \textrm{ps}(R)$, $[f]\not=hR$ and $1\leq i\leq n_f$.
	\end{remark}
\bibliographystyle{amsplain}

\end{document}